\documentclass[a4paper, final]{article}
\usepackage{amsfonts, amssymb, amsthm, amsmath}
\usepackage[utf8]{inputenc}
\usepackage[english]{babel}
\usepackage{graphicx}
\usepackage[algoruled]{algorithm2e}
\usepackage{booktabs}
\usepackage{rotating} 
\usepackage{hyperref}
\usepackage{fontawesome}
\usepackage{enumerate}

\newcommand{\natbasis}{\overline\phi}
\newcommand{\R}{\mathbb{R}}

\newcommand{\norm}[1]{\left\lVert#1\right\lVert}
\newcommand{\abs}[1]{\left\lvert#1\right\lvert}
\newcommand{\Pol}[2]{\mathcal{P}_{#2}^{#1}(\mathbb{R})}

\newtheorem{theorem}{Theorem}[section]
\newtheorem{lemma}[theorem]{Lemma}

\newtheorem{assumption}[theorem]{Assumption}
\newtheorem{remark}[theorem]{Remark}
\newtheorem{defn}[theorem]{Definition}

\title{On complexity constants of linear and quadratic models for derivative-free trust-region algorithms\footnote{This study was financed in part by the Coordenação de Aperfeiçoamento de Pessoal de Nível Superior – Brasil (CAPES) – Finance Code 001}} %
\author{ %
    \href{https://orcid.org/0000-0003-4660-4611}{\includegraphics[scale=0.6]{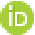}}
    A. E. Schwertner\footnote{Graduate Program in Mathematics, State University of Maringá, Paraná, Brazil} %
    \and
    \href{https://orcid.org/0000-0003-4963-0946}{\includegraphics[scale=0.6]{images/orcid_logo.eps}}
    F. N. C. Sobral\footnote{Department of Mathematics, State University
    of Maringá, Paraná, Brazil}
    \footnote{{\faEnvelopeO}~Corresponding author,
    \texttt{fncsobral@uem.br}}%
}

\begin{document}

\maketitle
	
\begin{abstract}
  Complexity analysis has become an important tool in the convergence analysis of optimization algorithms. For derivative-free optimization algorithms, it is not different. Interestingly, several constants that appear when developing complexity results hide the dimensions of the problem. This work organizes several results in
  literature about bounds that appear in derivative-free trust-region algorithms based on linear and quadratic models. All the constants are given explicitly by the quality of the sample set, dimension of the problem and number of sample points. We extend some results to allow ``inexact'' interpolation sets. We also provide a clearer proof than those already existing in literature for the underdetermined case.
\end{abstract}

\noindent
\textbf{Keywords}: Polynomial interpolation, Derivative-free trust-region algorithms, Minimum Frobenius norm, Error bounds

\section{Introduction}\label{sec:introduction}

Polynomial interpolation is one of the main aspects of model-based derivative-free algorithms. Given a function $f: \R^n \to \R$ and a set of interpolation points $\mathcal{Y} = \{y^0, \dots, y^p\}$, the interpolating polynomial $\mathrm{m} \in \Pol{a}{n}$ is such that
\begin{equation} \label{eq_interpolation} %
\mathrm{m}(y^i) = f(y^i),\quad i = 0, \dots, p,
\end{equation}
where $\Pol{a}{n}$ is the space of polynomials of degree at most $a$ in $n$ variables. It is well understood that quadratic polynomials ($a = 2$) provides a good approximation, being able to capture the curvature of $f$ using a reasonable amount of interpolation points~\cite[p.35]{CSV2009}.

More recently, the wort-case complexity analysis of optimization algorithms became popular. Although most of the bounds obtained are very pessimistic, the analysis provide further insights into the parameters and the dependence on the problems' dimensions.

When studying worst-case complexity of derivative-free trust-region algorithms using linear or quadratic polynomials, one has to deal with bounds related to the geometric quality of $\mathcal{Y}$ and to the Hessian of the model (in the quadratic case). In~\cite{CSV2009}, several bounds were provided, which explicitly show the dependence on the problems' dimensions. Such bounds were used, for example, in the excellent complexity analysis of~\cite{CR2019, GJV2016}. The complexity bounds generated the desire of performing derivative-free optimization in smaller subspaces, to further decrease the impact of dimensionality in the construction of models~\cite{CR2021}.

In~\cite{VKPS2017}, some error bounds were obtained for a relaxed
version of~\eqref{eq_interpolation}
\begin{equation} \label{eq_r_interpolation} %
|\mathrm{m}(y^i) - f(y^i)| \le \kappa \delta^2,
\end{equation}
where $\kappa$ is a constant and $\delta$ can be viewed as the
precision of the model~\cite{P2015,VKPS2017} or just as the
trust-region
radius~\cite{CSV2009}. Condition~\eqref{eq_r_interpolation} is related
to linear and quadratic models build from support vector regression
strategies. Unfortunately, only the determined case is analyzed
in~\cite{VKPS2017}, that is, the case where $p = dim(\Pol{a}{n}) -
1$. In the case of quadratic models, that would require $O(n^2)$
interpolation points in $\mathcal{Y}$. This work intends to extend the
bounds obtained in~\cite{VKPS2017} to the underdetermined case, which
is known to have better performance in practical implementations.

The main contributions of the present work are
\begin{itemize}
    \item the bounds associated with linear and quadratic polynomial models are organized, making clear their dependence on the dimension and quality of the interpolation set;
    \item the bounds from~\cite{VKPS2017} are extended to the underdetermined case.
\end{itemize}

This work is organized as follows.

Section~\ref{sec:determined} is concerned with bounds to linear and
quadratic polynomial models which are uniquely determined by the
sample set. Section~\ref{sec:underdetermined} is the main
contribution of this work and discusses error bounds for
underdetermined quadratic models, since they are the most efficient
models in practice~\cite{P2006}. Finally, in
Section~\ref{sec:conclusions} we state and organize all the bounds and
discuss possible open questions.

Throughout this work, we will refer to the interpolating polynomial $\mathrm{m}$ as the one satisfying the relaxed condition~(\ref{eq_r_interpolation}), $\overline{B}(x, \delta)$ is the closed ball centered at $x$ with radius $\delta$, $\norm{\cdot}$ is the Euclidean vector norm or its respective induced norm for matrices, $\mathcal{Y} = \{y^0, \dots, y^p\}$ is the set of $p + 1$ sample points, for which $f$ is assumed to be known, $\Pol{a}{n}$ is the space of polynomials of degree at most $a$ in $n$ variables, and $q = dim(\Pol{a}{n}) - 1$.

\section{Determined models} \label{sec:determined}

Let us assume that $\mathcal{Y} \subset \overline{B}(y^{0}, \delta)$. Determined interpolation models need that the number of sample points in $\mathcal{Y}$ equals the dimension of $\Pol{a}{n}$, that is, $p = q$ in our notation. In the linear case, we assume that $p = n = q$ and define matrices
\begin{equation}\label{eq_def_matriz_linear}
    \mathbf{L}_{L} = \left[\begin{array}{c}
            (y^{1} - y^{0})^{T} \\
            \vdots              \\
            (y^{n} - y^{0})^{T}
            \end{array}\right] = 
            \left[\begin{array}{ccc}
            y_{1}^{1} - y_{1}^{0} & \dots  & y_{n}^{1} - y_{n}^{0} \\
            \vdots                & \ddots & \vdots                \\
            y_{1}^{n} - y_{1}^{0} & \dots  & y_{n}^{n} - y_{n}^{0}
            \end{array}\right],
\end{equation}
and
\begin{equation}\label{eq_def_matriz_linear_escala}
    \hat{\mathbf{L}}_{L} = \frac{1}{\delta}\mathbf{L}_{L}.
\end{equation}
In the quadratic case, we assume $p = (n^{2} + 3n) / 2 = q$ and define the matrix
\begin{equation}\label{def_matriz_quadratica}
    \mathbf{L}_{Q} = \left[\begin{array}{c}
            \big(\overline{\varphi}(y^{1} - y^{0})\big)^{T} \\
            \vdots              \\
            \big(\overline{\varphi}(y^{q} - y^{0})\big)^{T}
            \end{array}\right],
\end{equation}
where
$\overline{\varphi}(x) = \bigg[x_{1}, x_{2}, \dots, x_{n},
\frac{1}{2}x_{1}^{2}, x_{1}x_{2}, x_{1}x_{3}, \dots, x_{1}x_{n},
\frac{1}{2}x_{2}^{2}, \dots, x_{n-1}x_{n}, \frac{1}{2}x_{n}^{2}
\bigg]^{T}$ is the vector whose elements form the natural basis of
monomials in $\Pol{2}{n}$, as defined in \cite[Section
3.1]{CSV2009}. We also consider the matrix
\begin{equation}\label{eq_def_matriz_quadratica_escala}
    \hat{\mathbf{L}}_{Q} = \mathbf{L}_{Q}
        \left[\begin{array}{cc}
        \mathbf{D}_{L}^{-1} & 0          \\
        0          & \mathbf{D}_{Q}^{-1},
        \end{array}\right]
\end{equation}
where $\mathbf{D}_{L} = \delta \mathbf{I}_{n \times n}$ e
$\mathbf{D}_{Q} = \delta^{2} \mathbf{I}_{(q - n) \times (q -
  n)}$. Matrices $\mathbf{L}_L$ and $\mathbf{L}_Q$ are related with the construction of
determined interpolation models,
while~\eqref{eq_def_matriz_linear_escala}
and~\eqref{eq_def_matriz_quadratica_escala} are their respective
scaled versions, mainly used for theoretical purposes.

We will assume that the following assumptions are valid.

\begin{assumption} \label{hip_f_diferenciavel}
$\nabla f$ is Lipschitz continuous with constant $L$ in a sufficiently large open bounded domain $\mathcal{X}\subset \mathbb{R}^{n}$.
\end{assumption}

\begin{assumption} \label{hip_f_limitada}
$f$ is bounded below.
\end{assumption}

\begin{assumption}\label{hip_L_invertivel_limitada}
If the model to be built is linear, then the matrix $\mathbf{L}_{L}$ is non-singular and there is a constant $\kappa_{L} > 0$ such that $\norm{\hat{\mathbf {L}}_{L}^{-1}} \leq \kappa_{L}$. If the model is quadratic, then the matrix $\mathbf{L}_{Q}$ is nonsingular and there is a constant $\kappa_{Q} > 0$ such that $\norm{\hat{\mathbf{L} }_{Q}^{-1}} \leq \kappa_{Q}$.
\end{assumption}

The next assumption is related to the relaxed concept of interpolation
models considered in this work.

\begin{assumption}\label{hip_dif_limitada}
There is a constant $\kappa \geq 0$ such that for every $y^{j} \in \mathcal{Y} \subset \overline{B}(y ^{0}, \delta)$
\begin{equation*}
    \abs{\mathrm{m}(y^{j}) - f{(y^{j})}} \leq \kappa \delta^{2}.
\end{equation*}
\end{assumption}

Note that Assumption~\ref{hip_dif_limitada} includes the creation of ``exact'' interpolation models and also models under uncertainty~\cite{VKPS2017}, under numerical errors or even when we can control the precision in which $f$ is calculated~\cite{BKM2022}. The next theorem provides the error bounds for determined polynomial interpolation under Assumption~\ref{hip_dif_limitada}. Its proof is provided in~\cite{VKPS2017}.

\begin{theorem}\label{teorema_limitantes_linear_quadratico}
Suppose that Assumptions \ref{hip_f_diferenciavel}, \ref{hip_f_limitada}, \ref{hip_L_invertivel_limitada} and \ref{hip_dif_limitada} hold. Then, for all $x \in \mathcal{X} \cap \overline{B}(y^{0}, \delta)$, if the model is linear
\begin{align*}
  \norm{\nabla^{2}\mathrm{m}(x)} & = 0,
  \\ 
  \norm{\nabla f(x) - \nabla \mathrm{m}(x)} &\leq  \left(L +  \left( \frac{1}{2} L + 2 \kappa \right) \kappa_L \sqrt{n} \right) \delta,
  \\ 
        \abs{f(x) - \mathrm{m}(x)} &\leq \left( \frac{1}{2} L + \kappa + \left( \frac{1}{2} L + 2 \kappa \right)  \kappa_L \sqrt{n} \right) \delta^{2},
\end{align*}
and, if the model is quadratic,
\begin{align*}
        \norm{\nabla^{2}\mathrm{m}(x)} & \leq 2 \kappa_Q \sqrt{2q} (\kappa + L), \\ 
        \norm{\nabla f(x) - \nabla \mathrm{m}(x)} &\leq  \left(2 \kappa_Q \sqrt{q} \left(1 + \sqrt{2} \right) (\kappa + L) \right) \delta, \\ 
        \abs{f(x) - \mathrm{m}(x)} &\leq \left(\frac{1}{2} L + \kappa + \kappa_Q \sqrt{q} \left(2 + 3\sqrt{2} \right) (\kappa + L) \right) \delta^{2}.
\end{align*}
\end{theorem}

In order to ensure Assumption \ref{hip_L_invertivel_limitada}, we need to make additional hypotheses about the geometry of the sample set. The existence and uniqueness of model $\textrm{m}$, also known as \emph{poisedness}, is not enough. In linear interpolation, for example, this can be achieved by assuming
that the points in $\mathcal{Y}$ are affinely linear independent~\cite{W2008}. Here, we assume that the set $\mathcal{Y}$ is $\Lambda$-poised on the ball $\overline{B}(y^{0}, \delta)$, for some $\Lambda > 0$. This definition is taken from~\cite{CSV2009} and will be better discussed in Section~\ref{sec:underdetermined} for
underdetermined polynomials. The following lemma, adapted from~\cite{VKPS2017}, shows that Assumption
\ref{hip_L_invertivel_limitada} is valid when the models to be built are linear or quadratic and $\mathcal{Y}$ is $\Lambda$-poised.

\begin{lemma}\label{lema_08}
Let $\Lambda > 0$. If $\mathcal{Y} = \{y^{0}, y^{1}, \dots, y^{n}\}$ is $\Lambda$-poised in $\overline{B}(y^{0}, \delta)$ with respect to the basis $\phi$ of $\mathcal{P}^{1}_{n}$, then
$\norm{\hat{\mathbf{L}}_{L}^{-1}} \leq \Lambda\sqrt{n}$. If $\mathcal{Y} = \{y^{0}, y^{1}, \dots, y^{q}\}$ is a
$\Lambda$-poised set in $\overline{B}(y^{0}, \delta)$ with respect to the basis $\phi$ in $\mathcal{P}^{2}_{n}$, then, $ \norm{\hat{\mathbf{L}}_{Q}^{-1}} \leq 4\Lambda\sqrt{(q+1)^{3}}$.
\end{lemma}

\section{Underdetermined models} \label{sec:underdetermined}

The main drawback of determined polynomials is the high number of interpolation points needed in the quadratic case. One could use only linear polynomials, but quadratic polynomials are known to better explore the curvature of $f$. Therefore, we are interested in building quadratic polynomials using less than $q$ interpolation points. Now we assume that $n < p < q$ and $q = dim(\mathcal{P}_{n}^{2}(\mathbb{R})) - 1 = (n^{2} + 3n) / 2$.

When dealing with classical interpolation theory, in the sense of (\ref{eq_interpolation}), building an underdetermined interpolation quadratic model $\mathrm{m} \in \Pol{2}{n}$ can be viewed as finding $\alpha \in \mathbb{R}^{q + 1}$ such that
\begin{equation*}
    \sum_{j = 0}^{q} \alpha_{j}\phi_{j}(y^{i}) = f(y^{i}) \text{, for } i = 0, \dots, p,
\end{equation*}
or, equivalently,
\begin{equation} \label{def_sistema_interpolacao}
    \textbf{M}(\phi, \mathcal{Y}) \alpha = f(\mathcal{Y}),
\end{equation}
where $\phi = \{\phi_{0}(x), \phi_{1}(x), \dots, \phi_{q}(x) \}$ is a basis for $\Pol{2}{n}$,
\begin{equation} \label{def_matriz_M_subdeterminada}
    \textbf{M}(\phi, \mathcal{Y}) = 
    \left[
        \begin{array}{cccc}
            \phi_{0}(y^{0}) & \phi_{1}(y^{0}) & \cdots & \phi_{q}(y^{0}) \\
            \phi_{0}(y^{1}) & \phi_{1}(y^{1}) & \cdots & \phi_{q}(y^{1}) \\
            \vdots          & \vdots          & \ddots & \vdots          \\
            \phi_{0}(y^{p}) & \phi_{1}(y^{p}) & \cdots & \phi_{q}(y^{p})
        \end{array}
    \right], \text{ and }
    f(\mathcal{Y}) = 
    \left[
        \begin{array}{c}
            f(y^{0}) \\
            f(y^{1}) \\
            \vdots      \\
            f(y^{p}) 
        \end{array}
    \right].
\end{equation}
Under the above conditions, the number of interpolation points is less than the dimension of the space of the polynomials $\mathcal{P}_{n}^{2}(\mathbb{R})$, that is, $\abs{\mathcal{Y}} = p + 1 < q + 1 = dim(\mathcal{P}_{n}^{2}(\mathbb{R}))$. Therefore, the interpolation polynomials defined by~(\ref{def_sistema_interpolacao}) are no longer unique.

In this section, our goal is to build an underdetermined quadratic model, which approximately interpolates the function $f$ over $ \mathcal{Y} \subset \overline{B}(y^{0}, \delta)$, in the sense of Assumption~\ref{hip_dif_limitada}. First, let us consider matrices 
\begin{equation}\label{def_matriz_quad_sub}
    \mathbf{L}_{s} = \left[\begin{array}{c}
            (y^{1} - y^{0})^{T} \\
            \vdots              \\
            (y^{p} - y^{0})^{T}
            \end{array}\right] = 
            \left[\begin{array}{ccc}
            y_{1}^{1} - y_{1}^{0} & \dots  & y_{n}^{1} - y_{n}^{0} \\
            \vdots                & \ddots & \vdots                \\
            y_{1}^{p} - y_{1}^{0} & \dots  & y_{n}^{p} - y_{n}^{0}
            \end{array}\right],
\end{equation}
and
\begin{equation}\label{def_matriz_quad_sub_escalonada}
    \hat{\mathbf{L}}_{s} = \frac{1}{\delta}\mathbf{L}_{s}.
\end{equation}

In order to obtain the main error bounds for the underdetermined case, we need to assume some geometric conditions. In other words, we need the sample set $\mathcal{Y}$ to be poised for linear regression and the Hessian of the model to be bounded. Such properties are given by assumptions~\ref{hip_L_coluna_completo_limitada} and~\ref{hip_norma_H_limitada}.

\begin{assumption}\label{hip_L_coluna_completo_limitada}
The matrix $\mathbf{L}_{s} \in \mathbb{R}^{p \times n}$ has full column rank, that is, $rank(\mathbf{L}_{s}) = n $, and there is a constant $\kappa_{s} > 0$ such that $\norm{\hat{\mathbf{L}}^{\dagger}_{s}} \leq \kappa_{s}$, where $\hat{\mathbf{L}}^{\dagger}_{s}$ denotes the Moore-Penrose pseudo-inverse of the matrix $\hat{\mathbf{L}}_{s}$.
\end{assumption}

\begin{assumption}\label{hip_norma_H_limitada}
There is a constant $\kappa_{H} \geq 0$ such that $\norm{\mathbf{H}} \leq \kappa_{H}$.
\end{assumption}

The following theorem establishes error bounds for the underdetermined model and its derivatives in $\mathcal{X} \cap \overline{B}(y^{0}, \delta)$.

\begin{theorem}\label{teo_limitantes_subdeterminada}
Suppose that Assumptions \ref{hip_f_diferenciavel}, \ref{hip_f_limitada}, \ref{hip_L_coluna_completo_limitada} and \ref{hip_norma_H_limitada} hold. Then, for all $x \in \mathcal{X} \cap \overline{B}(y^{0}, \delta)$,
\begin{align}
        \abs{f(x) - \mathrm{m}(x)} &\leq \bigg( \frac{1}{2} \big( L + \kappa_{H} \big) + \kappa +  2\kappa_{s} \sqrt{p} \bigg(L + \kappa + \frac{3}{4} \kappa_{H} \bigg)\bigg) \delta^{2}, \label{teo_lim_sub_01} \\
        \norm{\nabla f(x) - \nabla \mathrm{m}(x)} &\leq 2 \kappa_{s} \sqrt{p}\bigg(L + \kappa +  \frac{3}{4}\kappa_{H} \bigg)\delta. \label{teo_lim_sub_02}
\end{align}

\end{theorem}
\begin{proof}
The proof is strongly based in the ideas of~\cite[Theorem 5.4]{CSV2009} and~\cite[Theorem 2]{VKPS2017}. In addition $\norm{\mathbf{H}}$ and $\norm{\hat{\mathbf{L}}^{\dagger}_{s}}$ were replaced by their respective bounds in equations~\eqref{teo_lim_sub_01} and~\eqref{teo_lim_sub_02}.
\end{proof}

\subsection{Minimum Frobenius norm models}

To establish a bound for the norm of the Hessian of the model, we must specify its construction. In this sense, we will reduce the degree of freedom in choosing the model, assuming that it is the minimum Frobenius norm model. We will also need to define what is a relaxed minimum Frobenius norm model.

Let $\overline{\phi}$ be the natural basis of $\mathcal{P}^{2}_{n}(\mathbb{R})$ and consider its split into the sets $\overline{\phi}_{L} = \{ 1, x_{1}, \dots, x_{n} \}$ and $\overline{\phi}_{Q} = \big\{ \frac{1}{2} x_{1}^{2}, x_{1}x_{ 2}, \dots, \frac{1}{2} x_{n}^{2} \big \}$. Then we can write the underdetermined interpolation model as
\begin{equation*}
    \mathrm{m}(x) = \alpha_{L}^{T} \overline{\phi}_{L}(x) + \alpha_{Q}^{T} \overline{\phi}_{Q}(x),
\end{equation*}
where $\alpha_{L}$ and $\alpha_{Q}$ are appropriate parts of the vector of coefficients $\alpha$. As in \cite[p. 80-81]{CSV2009}, we can define the minimum Frobenius norm solution $\alpha^{\textrm{mfn}}$ as the solution of the optimization problem,
\begin{equation}\label{eq_definicao_alpha_mfn}
    \begin{array}{cc}
        \min & \frac{1}{2}\norm{\alpha_{Q}}^{2} \\
        \text{s.t.} & \mathbf{M}(\overline{\phi}_{L}, \mathcal{Y})\alpha_{L} + \mathbf{M}(\overline{\phi}_{Q}, \mathcal{Y})\alpha_{Q} = f(\mathcal{Y}),
    \end{array}
\end{equation}
in the variables $\alpha_{L}$ and $\alpha_{Q}$, where $\mathcal{Y}$ is the set of interpolation points and the matrices $\mathbf{M}(\overline{\phi}_ {L}, \mathcal{Y})$ and $\mathbf{M}(\overline{\phi}_{Q}, \mathcal{Y})$ are submatrices of $\mathbf{M}(\overline{ \phi}, \mathcal{Y})$, given in (\ref{def_matriz_M_subdeterminada}), whose columns correspond to the elements of $\overline{\phi}_{L}$ and $\overline{\phi}_{ Q}$, respectively.

It is important to point out that the condition of existence and uniqueness of the minimum Frobenius norm model is that the matrix defined by
\begin{equation*}
    \mathbf{F}(\overline{\phi}, \mathcal{Y}) = \left[\begin{array}{cc}
            \mathbf{M}(\overline{\phi}_{Q}, \mathcal{Y}) \mathbf{M}(\overline{\phi}_{Q}, \mathcal{Y})^{T} & \mathbf{M}(\overline{\phi}_{L}, \mathcal{Y}) \\
            \mathbf{M}(\overline{\phi}_{L}, \mathcal{Y})^{T} & 0
            \end{array}\right]
\end{equation*}
is nonsingular. Note that $\mathbf{F}(\overline{\phi}, \mathcal{Y})$ is nonsingular if and only if $\mathbf{M}(\overline{\phi}_{L}, \mathcal {Y})$ has a full column rank and $\mathbf{M}(\overline{\phi}_{Q}, \mathcal{Y}) \mathbf{M}(\overline{\phi}_{Q} , \mathcal{Y})^{T}$ is positive definite on the nullspace of $\mathbf{M}(\overline{\phi}_{L}, \mathcal{Y})^{T} $, this last condition being guaranteed if $\mathbf{M}(\overline{\phi}_{Q}, \mathcal{Y})$ has full row rank \cite[p. 81]{CSV2009}.

We say that $\mathcal{Y}$ is poised in the minimum Frobenius norm sense if the matrix $\mathbf{F}(\overline{\phi}, \mathcal{Y})$ is nonsingular. Note that poisedness in the minimum Frobenius norm sense implies poisedness in the linear interpolation or regression senses and, as a result, poisedness for underdetermined quadratic interpolation in the minimum norm sense \cite[p. 81]{CSV2009}.

\begin{remark}
As we saw earlier, if $\mathcal{Y}$ is poised in the minimum Frobenius norm sense, then $\mathbf{F}(\overline{\phi}, \mathcal{Y})$ is nonsingular and hence $ \mathbf{M}(\overline{\phi}_{L}, \mathcal{Y})$ has a full column rank. So, as
\begin{equation*}
    \mathbf{M}(\overline{\phi}_{L}, \mathcal{Y}) =
    \begin{bmatrix}
    1 & 0 \\
    e & \mathbf{Id}
    \end{bmatrix}
    \left[\begin{array}{cc}
         1 & {y^0}^T \\
         0 & \mathbf{L}_{s}
    \end{array}\right] = \mathbf{E}^{-1}
    \left[\begin{array}{cc}
         1 & {y^0}^T \\
         0 & \mathbf{L}_{s}
    \end{array}\right]
\end{equation*}
where $e = [1, 1, \dots, 1]^{T} \in \mathbb{R}^{p}$, and $\mathbf{Id} \in \mathbb{R}^{p \times p}$ is the identity matrix. Since $\mathbf{E}$ is an elementary nonsingular matrix and $ \mathbf{M}(\overline{\phi}_{L}, \mathcal{Y})$ has full column rank, it follows that $\mathbf{L}_{s}$ has full column rank, which partially satisfies Assumption \ref{hip_L_coluna_completo_limitada}.
\end{remark}

Therefore, we will assume that the following hypothesis holds.

\begin{assumption}\label{hip_posic_inter_linear}
The set of sample points $\mathcal{Y} = \{ y^{0}, y^{1}, \dots, y^{p} \} \subset \mathbb{R}^{n}$ is poised in the minimum Frobenius norm sense on $\overline{B}(y^{0}, \delta)$.
\end{assumption}

Consider the following auxiliary lemma.

\begin{lemma}\label{lema_aux_01}
There is a number $\sigma_{\infty} > 0$ such that for any choice of $v$ satisfying $\norm{v}_{\infty} = 1$, there is $z \in \overline{B}(0 , 1)$ such that $\abs{v^{T}\phi(z)}\geq \sigma_{\infty}$. If, in addition, $v^{T}\overline{\phi}(x)$ is a quadratic polynomial and $\overline{\phi}$ is the natural basis in $\Pol{2}{n}$, then
\begin{equation*}
    \max_{x \in \overline{B}(0,1)} \abs{v^{T}\overline{\phi}(x)} \geq \frac{1}{4}.
\end{equation*}
If $\overline{\phi}$ is the natural basis of $\Pol{1}{n}$, then,
\begin{equation*}
     \max_{x \in \overline{B}(0, 1)} \abs{v^{T}\overline{\phi}(x)} \geq 1.
\end{equation*}
\end{lemma}
\begin{proof}
See lemmas 3.10, 3.11, and 6.7 in~\cite{CSV2009}.
\end{proof}

\begin{remark}\label{obs_01}
Let $\overline{\phi}$ be the natural basis of $\Pol{1}{n}$. Given $\overline{v} \in \mathbb{R}^{n+1}$ with $\norm{\overline{v}} = 1$, there are $\beta \in (0, \sqrt{n+1})$ and $v \in \mathbb{R}^{n+1}$ such that $\norm{v}_{\infty} = 1$ and $v = \beta \overline{v}$. Thus, by Lemma \ref{lema_aux_01}, we have that
\begin{equation*}
    \max_{x \in \overline{B}(0, 1)} \abs{\overline{v}^{T}\overline{\phi}(x)} = \max_{x \in \overline{B}(0, 1)} \frac{1}{\beta} \abs{v^{T}\overline{\phi}(x)} \geq \frac{1}{\sqrt{n+1}} \max_{x \in \overline{B}(0, 1)} \abs{v^{T}\overline{\phi}(x)} \geq \frac{1}{\sqrt{n+1}}.
\end{equation*}
\end{remark}

Equation~\eqref{eq_definicao_alpha_mfn} is strongly related with the exact interpolation condition~\eqref{eq_interpolation}. In order to relax such condition we will first define minimum Frobenius norm Lagrange polynomials.

\begin{defn} \label{def_mfn_lag_poly}
Given a set $\mathcal{Y}$ of $p + 1$ interpolation points, the set of $p + 1$ polynomials $\ell_j(x) = \alpha_j^T \natbasis(x)$, $j = 0, \dots, p$, is called set of minimum Frobenius norm Lagrange polynomials for $\natbasis$, if $\alpha_j = ((\alpha_j)_L, (\alpha_j)_Q)$ is the solution of
\[
\begin{array}{ll}
    \min & \frac{1}{2} \norm{\alpha_Q}^2 \\
    \mbox{s. t.} & M(\natbasis_L, \mathcal{Y}) \alpha_L + M(\natbasis_Q, \mathcal{Y}) \alpha_Q = e_{j + 1},
\end{array}
\]
where $e_j$ is the $j$-th canonical vector in $\R^{p + 1}$.
\end{defn}

By comparing the KKT conditions of Definition~\ref{def_mfn_lag_poly} and equation~\eqref{eq_definicao_alpha_mfn} it is possible to show that the minimum Frobenius norm polynomial that interpolates $f$ in $\mathcal{Y}$ can be defined as $\mathrm{m}(x) = \sum_{j = 1}^{p + 1} \ell_j(x) f(y^j).  $ So, in order to relax~\eqref{eq_definicao_alpha_mfn}, we will say that our relaxed minimum Frobenius norm polynomials are given by
\begin{equation}\label{eq_relaxed_mfn_pol}
    \mathrm{m}(x) = \sum_{j = 1}^{p + 1} \ell_j(x) \gamma_j
\end{equation}
where $\gamma_j$, $j = 1, \dots, p + 1$, are chosen such that Assumption~\ref{hip_dif_limitada} holds. One can think that $\mathrm{m}$ was obtained by solving~\eqref{eq_definicao_alpha_mfn} with $f(y^j)$ replaced by $\gamma_j$ for all $j = 0, \dots, p$.

The theorems that follow are the main contribution of this work. Their establish bounds for $\norm{\nabla^2 m(x)}$ and $\norm{\hat L^\dagger_s}$ under the hypothesis of well poised interpolation sets and, therefore, assumptions~\ref{hip_L_coluna_completo_limitada} and~\ref{hip_norma_H_limitada} can be removed from Theorem~\ref{teo_limitantes_subdeterminada}. Using a partial definition from~\cite{CSV2009}, we first define what we mean as well poised interpolation sets for minimum Frobenius norm quadratic models.

\begin{defn} \label{def_mfn_lambda_poised} %
Let $\Lambda > 0$ and $B \subseteq \R^n$ be given.  Let $\natbasis$ be the natural basis of monomials of $\Pol{2}{n}$. A poised set $\mathcal{Y}$ is said to be $\Lambda$-poised in $B$ in the minimum Frobenius norm sense if and only if for any $x \in B$, the solution $\lambda(x) \in \R^{p + 1}$ of
\begin{equation*}
    \begin{array}{ll}
    \min & \frac{1}{2} \norm{M(\natbasis_Q, \mathcal{Y})^T \lambda(x) - \natbasis_Q(x)}^2 \\
    \mbox{s.t.} & M(\natbasis_L, \mathcal{Y})^T \lambda(x) = \natbasis_L(x)
    \end{array}
\end{equation*}
is such that $\norm{\lambda(x)}_\infty \le \Lambda$.
\end{defn}

Note that Definition~\ref{def_mfn_lambda_poised} does not need to be
relaxed. We are now ready to provide the necessary bounds in terms of
the quality of the interpolation set, the Lipschitz constant and the
dimensions of the problem.

\begin{theorem}\label{teo_hessiana_lim_sub}
Suppose that Assumptions \ref{hip_f_diferenciavel} and \ref{hip_posic_inter_linear} hold. Assume that the set $\mathcal{Y} = \{y^{0}, y^{1}, \dots, y^{p} \}$ is $\Lambda$-poised in $\overline{B} (y^{0}, \delta)$ in the minimum Frobenius norm sense and $\delta_{\max} > 0$ is an upper bound for $\delta$. Then,
\begin{equation*}
    \norm{\nabla^{2}\mathrm{m}(x)} \leq \left(\kappa + \frac{L}{2} \right) \frac{4 \Lambda (p + 1) \sqrt{2(q+1)}}{c(\delta_{\max})^{2}},
\end{equation*}
where $c(\delta_{\max}) = \min \{ 1, 1 / \delta_{\max}, 1 / \delta_{\max}^{2} \}$.
\end{theorem}
\begin{proof}
This proof follows the proof of~\cite[Theorem 5.7]{CSV2009}. Assume, without loss of generality, that $y^{0} = 0$. By Lemma \ref{lema_aux_01}, the definition of $\Lambda$-poisedness and arguments very similar to~\cite[Theorem 5.7]{CSV2009} we have that
\begin{equation} \label{teo_hessiana_lim}
    \norm{\nabla^{2} \ell_{j}(x)} \leq \sqrt{2(q+1)}\frac{4\Lambda}{\delta^{2}c(\delta_{\max})^{2}},\quad j = 0, \dots, p,
\end{equation}
where $\delta_{\max}$ and $c(\delta_{\max})$ were defined by the theorem.
  
Now, let us consider the function 
\begin{equation*}
    \hat{f}(x) = f(x) - f(y^{0}) - \nabla f(y^{0})^{T}(x - y^{0}).
\end{equation*}
Note that $\hat{f}(y^{0}) = 0$, $\nabla \hat{f}(y^{0}) = 0 $ and the Hessian remains unchanged. In addition, if $\mathrm{m}(x)$ is a relaxed minimum Frobenius norm model for $f$ over the set $\mathcal{Y}$, then it is easy to see that
\begin{equation*}
    \hat{\mathrm{m}}(x) = \mathrm{m}(x) - f(y^{0}) - \nabla f(y^{0})^{T}(x - y^{0})
\end{equation*}
also satisfies Assumption~\ref{hip_dif_limitada} for $\hat f$ and the points in $\mathcal{Y}$.  Therefore, we can assume without loss of generality that $f(y^{0}) = 0$ and $\nabla f(y^{0}) = 0$. Thus, we have that $\abs{ f(x) } = \abs{ f(x) - f(y^{0}) - \nabla f(y^{0})^{T}(x - y^{0}) } \leq (L/2) \delta^{2} $, for all $x \in B(y^{0}, \delta)$ and, by~\eqref{eq_relaxed_mfn_pol}, we can prove that
\begin{equation} \label{teo_hessiana_lim_sub_eq05}
    \abs{\gamma_j} \le (\kappa + L / 2) \delta^2, \quad j = 0, \dots, p.
\end{equation}

By using the definition of $\mathrm{m}(x)$ in~\eqref{eq_relaxed_mfn_pol} and bounds~(\ref{teo_hessiana_lim}) and~(\ref{teo_hessiana_lim_sub_eq05}), we finally get that
\begin{equation*}
    \begin{split}
        \norm{\mathbf{H}} & = \norm{\nabla^{2} \mathrm{m}(x)} = \norm{\sum_{j = 0}^{p} \gamma_j \nabla^{2} \ell_{j}(x) } \leq \sum_{j = 0}^{p} \abs{\gamma_j} \norm{\nabla^{2} \ell_{j}(x)} \\
        & \le \left(\kappa + \frac{L}{2} \right) \frac{4  \Lambda (p + 1) \sqrt{2(q+1)}}{c(\delta_{\max})^{2}},
    \end{split}
\end{equation*}
and the proof is complete.
\end{proof}

\begin{theorem}\label{lema_lim_matriz_quad_sub}
  If $\mathcal{Y} = \{y^{0}, y^{1}, \dots, y^{p}\}$ is a
  $\Lambda$-poised set in $\overline{B}(y ^{0}, \delta)$ in the sense
  of the minimum Frobenius norm, then
  \begin{equation}
    \norm{\hat{\mathbf{L}}_{s}^{\dagger}} \leq \Lambda
    \sqrt{2(n + 1)} (p + 1).    
  \end{equation}
\end{theorem}
\begin{proof}
It is known that $\Lambda$-poisedness does not depend on the scale of the sample set and it is invariant with respect to shifts~\cite{CSV2009}. Therefore, let us consider the set
\begin{equation*}
    \hat{\mathcal{Y}} = \bigg\{0, \frac{y^{1} - y^{0}}{\delta}, \dots, \frac{y^{p} - y^{0}}{\delta} \bigg\},
\end{equation*}
which is $\Lambda$-poised in $\overline{B}(0, 1)$, and the matrices given by
\begin{equation*}
    \hat{\mathbf{M}}_{L} = \mathbf{M}(\overline{\phi}_{L}, \hat{\mathcal{Y}}) = 
    \left[\begin{array}{cc}
         1 & 0 \\
         e & \hat{\mathbf{L}}_{s}
    \end{array}\right],
    \mathbf{E} = 
    \left[\begin{array}{cc}
         1 & 0 \\
         -e & \mathbf{Id}
    \end{array}\right],
    \mathbf{E}^{-1} = 
    \left[\begin{array}{cc}
         1 & 0 \\
         e & \mathbf{Id}
    \end{array}\right],  \text{ and }
\end{equation*}
\begin{equation*}
    \mathbf{Q} = \mathbf{E} \hat{\mathbf{M}}_{L} = 
    \left[\begin{array}{cc}
         1 & 0 \\
         0 & \hat{\mathbf{L}}_{s}
    \end{array}\right],
\end{equation*}
where $\hat{\mathbf{M}}_{L} \in \mathbb{R}^{(p + 1) \times (n + 1)}$, $\mathbf{E} \in \mathbb{ R}^{(p + 1) \times (p + 1)}$, $\mathbf{E}^{-1} \in \mathbb{R}^{(p + 1) \times (p + 1 )}$, $\mathbf{Q} \in \mathbb{R}^{(p + 1) \times (n + 1)}$, $e = [1, 1, \dots, 1]^{T } \in \mathbb{R}^{p}$, $\mathbf{Id} \in \mathbb{R}^{p \times p}$ is the identity matrix and $\hat{\mathbf{L}} _{s}$ was defined in (\ref{def_matriz_quad_sub_escalonada}). Note that the Moore-Penrose pseudoinverse of $\mathbf{Q}$ is given by
\begin{equation*}
    \mathbf{Q}^{\dagger} = 
        \left[\begin{array}{cc}
            1                  & 0 \\
            0 & \hat{\mathbf{L}}_{s}^{\dagger}
        \end{array}\right],
\end{equation*}
where $\mathbf{Q}^{\dagger} \in \mathbb{R}^{(n + 1) \times (p + 1)}$. Thus,
\begin{equation}\label{lema_lim_matriz_quad_sub_eq_01}
    \norm{\mathbf{Q}^{\dagger}} = \max \Big\{1, \norm{\hat{\mathbf{L}}_{s}^{\dagger}} \Big\} \geq \norm{\hat{\mathbf{L}}_{s}^{\dagger}}.
\end{equation}

By Definition~\ref{def_mfn_lambda_poised}, for every $x \in \overline {B}(0, 1) \subset \mathbb{R}^{n}$ there exists $\lambda(x) \in \mathbb{R}^{p+1}$, with $\norm{\lambda( x)}_{\infty} \leq \Lambda$, such that $\hat{\mathbf{M}}_{L}^{T} \lambda(x) = \overline{\phi}_{L} (x)$. Since $\mathbf{E}$ is a non-singular matrix, we have that for each $x \in \overline{B}(0, 1)$, 
\begin{equation}\label{lema_lim_matriz_quad_sub_eq_02}
    \begin{aligned}
        \hat{\mathbf{M}}_{L}^{T} \lambda(x) = \overline{\phi}_{L}(x) 
        &\Longleftrightarrow \hat{\mathbf{M}}_{L}^{T} \big( \mathbf{E}^{T} \mathbf{E}^{-T} \big) \lambda(x) = \overline{\phi}_{L}(x) \\
        &\Longleftrightarrow \big( \hat{\mathbf{M}}_{L}^{T} \mathbf{E}^{T} \big) \big(\mathbf{E}^{-T}  \lambda(x) \big) = \overline{\phi}_{L}(x) \\
        &\Longleftrightarrow \big( \mathbf{E} \hat{\mathbf{M}}_{L} \big)^{T} \big( \mathbf{E}^{-T} \lambda(x) \big) = \overline{\phi}_{L}(x) \\
        &\Longleftrightarrow \mathbf{Q}^{T} \big( \mathbf{E}^{-T} \lambda(x) \big) = \overline{\phi}_{L}(x) \\
        &\Longleftrightarrow \mathbf{E}^{-T} \lambda(x) = \mathbf{Q}^{\dagger^{T}} \overline{\phi}_{L}(x).
    \end{aligned}
\end{equation}
Note that $\norm{\lambda(x)}_{\infty} \leq \Lambda $ and 
$
    \mathbf{E}^{-T} \lambda(x) = \Bigg[ \sum_{j = 0}^{p} [\lambda(x)]_{j}, \lambda(x)^T \Bigg]^{T}
$.
Thus, from (\ref{lema_lim_matriz_quad_sub_eq_02}), we get
\begin{equation}\label{lema_lim_matriz_quad_sub_eq_03}
    \begin{aligned}
        \norm{\mathbf{Q}^{\dagger^{T}} \overline{\phi}_{L}(x)} 
        &= \norm{\mathbf{E}^{-T} \lambda(x)}
        = \sqrt{ \Bigg( \sum_{j = 0}^{p} [\lambda(x)]_{j} \Bigg)^{2} + \sum_{j = 1}^{p} [\lambda(x)]_{j}^{2} } \\
        &\leq \sqrt{ \Bigg( \sum_{j = 0}^{p} \abs{[\lambda(x)]_{j}} \Bigg)^{2} + \sum_{j = 1}^{p} \abs{[\lambda(x)]_{j}}^{2} } \\
        &\leq \sqrt{\big( (p + 1)\Lambda \big)^{2} + p \Lambda^{2}}
        < \sqrt{ 2( p + 1 )^{2} \Lambda^{2} } \\
        &= \Lambda \sqrt{2} (p + 1).
    \end{aligned}
\end{equation}

Let $\overline{v} \in \mathbb{R}^{n+1}$ be a singular right unit vector of $\mathbf{Q}^{\dagger^{T}}$ associated with the largest singular value $\sigma_{1}$, and $\overline{x} \in \mathbb{R}^{n}$ the maximizer of $\abs{\overline{v}^{T}\overline{\phi}_{L }(x)}$ in $\overline{B}(0, 1)$. Using the SVD decomposition and the fact that $\norm{\mathbf{Q}^\dagger} = \sigma_1$ it is not hard to show (see~\cite[Lemma 3.13]{CSV2009}) that
\begin{equation} \label{pseudoinversa_bound}
     \norm{\mathbf{Q}^{\dagger^{T}} \overline{\phi}_{L}(\overline{x})} \ge |\overline{v}^T \overline{ \phi}_L(\overline{x})| \norm{\mathbf{Q}^{\dagger^{T}}}.
\end{equation}
Then, from Lemma \ref{lema_aux_01}, Remark \ref{obs_01}, (\ref{lema_lim_matriz_quad_sub_eq_03}), and~\eqref{pseudoinversa_bound},
\begin{equation*}
    \begin{aligned}
        \Lambda \sqrt{2} (p + 1) 
        &\geq \norm{\mathbf{Q}^{\dagger^{T}} \overline{\phi}_{L}(\overline{x})}
        \geq \abs{\overline{v}^{T}\overline{\phi}_{L}(\overline{x})} \norm{\mathbf{Q}^{\dagger^{T}}} \\
        &= \abs{\overline{v}^{T}\overline{\phi}_{L}(\overline{x})} \norm{\mathbf{Q}^{\dagger}} 
        = \max_{x \in \overline{B}(0, 1)} \abs{\overline{v}^{T}\overline{\phi}_{L}(x)} \norm{\mathbf{Q}^{\dagger}} \\ &\geq \frac{1}{\sqrt{n+1}} \norm{\mathbf{Q}^{\dagger}}.
    \end{aligned}
\end{equation*}
Thus, from the previous inequality and from (\ref{lema_lim_matriz_quad_sub_eq_01}), it follows that 
\begin{equation*}
    \norm{\hat{\mathbf{L}}_{s}^{\dagger}} \leq \norm{\mathbf{Q}^{\dagger}} \leq \Lambda \sqrt{2 (n + 1)} (p + 1),
\end{equation*}
 and we complete the proof.
\end{proof}

\section{Conclusions}\label{sec:conclusions}

In this article, we organized several results from the literature about error bounds for linear and quadratic models related to derivative-free trust-region algorithms. We also extended the results relaxed results of~\cite{VKPS2017} to underdetermined models and provided a clearer proof than~\cite{CSV2009} for the bound on $\norm{\mathbf{L_s}^{\dagger}}$ in the minimum Frobenius norm case. Table~\ref{table_error_bounds} provides a compilation of the results presented here and is useful for future works in worst case complexity.

Future work may include the estimation of bounds for the Hessian of minimum norm underdetermined quadratic models. In \cite{CSV2008}, the authors obtain bounds for the projection of errors onto a specific linear subspace, which is not very useful in practical terms. On the other hand, \cite[p. 79]{CSV2009}, suggests that by using an overall poisedness constant for the sample set, it is possible to establish bounds for the model's Hessian. Another interesting question is whether models constructed by support vector regression~\cite{VKPS2017} using underdetermined quadratic polynomials are able to satisfy Assumption~\ref{hip_dif_limitada} and their practical benefits under noisy blackbox functions.

\begin{sidewaystable}
  \centering
    \begin{tabular}{llll}
    \toprule
    {\bf Model Type} & {\bf Error} & {\bf Error bound} & Ref. \\ \midrule
    Linear determined  & $\abs{\mathrm{m}(x) - f(x)}$ & $\left( \frac{1}{2} L + \kappa + \left( \frac{1}{2} L + 2 \kappa \right) \Lambda n \right) \delta^{2}$ & \cite{VKPS2017} \\
    & $\norm{\nabla \mathrm{m}(x) - \nabla f(x)}$ & $\left(L +  \left( \frac{1}{2} L + 2 \kappa \right) \Lambda n \right) \delta$ & \cite{VKPS2017} \\
    & $\norm{\nabla^2 \mathrm{m}(x)}$ & $0$ & \cite{VKPS2017} \\ \midrule
    Quadratic determined  & $\abs{\mathrm{m}(x) - f(x)}$ & $\left(\frac{1}{2} L + \kappa + 4 \Lambda \sqrt{q(q+1)^{3}} \left(2 + 3\sqrt{2} \right) (\kappa + L) \right) \delta^{2}$ & \cite{VKPS2017} \\
    & $\norm{\nabla \mathrm{m}(x) - \nabla f(x)}$ & $\left(8 \Lambda \sqrt{q(q+1)^{3}} \left(1 + \sqrt{2} \right) (\kappa + L) \right) \delta$ & \cite{VKPS2017} \\
    & $\norm{\nabla^2 \mathrm{m}(x)}$ & $8 \Lambda \sqrt{2q(q+1)^{3}} (\kappa + L)$ & \cite{VKPS2017} \\ \midrule
    Quadratic underdetermined  & $\abs{\mathrm{m}(x) - f(x)}$ & $\left(\frac{1}{2} \left( L + \kappa_{H} \right) + \kappa +  2\kappa_{s} \sqrt{p} \left(L + \kappa + \frac{3}{4} \kappa_{H} \right)\right) \delta^{2}$ &  This work \\
    & $\norm{\nabla \mathrm{m}(x) - \nabla f(x)}$ & $2 \kappa_{s} \sqrt{p}\left(L + \kappa +  \frac{3}{4}\kappa_{H} \right)\delta$ & This work \\ \midrule
      Min. Frobenius norm & $\abs{\mathrm{m}(x) - f(x)}$ & $\left(\frac{1}{2} \left(L + \left(\kappa + \frac{L}{2} \right)  \Lambda \frac{4 (p + 1)
    \sqrt{2(q+1)}}{c(\delta_{\max})^{2}} \right) + \kappa +\right.$ & This work \\
      & & $\quad + \left. 2 \Lambda \sqrt{2 p (n+1) } (p + 1) \left(L + \kappa + \left(\kappa + \frac{L}{2} \right)  \Lambda \frac{3 (p + 1)
    \sqrt{2(q+1)}}{c(\delta_{\max})^{2}} \right) \right) \delta^2$ & This work \\
    & $\norm{\nabla \mathrm{m}(x) - \nabla f(x)}$ & $2 \Lambda \sqrt{2 p (n+1) } (p + 1) \left(L + \kappa + \left(\kappa + \frac{L}{2} \right)  \Lambda \frac{3 (p + 1)
    \sqrt{2(q+1)}}{c(\delta_{\max})^{2}} \right)\delta$ & This work\\
    & $\norm{\nabla^2 \mathrm{m}(x)}$ & $\left(\kappa + \frac{L}{2} \right)  \Lambda \frac{4 (p + 1)
    \sqrt{2(q+1)}}{c(\delta_{\max})^{2}}$ & This work\\ \bottomrule
    \end{tabular}
  \caption{Error bounds for linear and quadratic interpolation models
    under Assumption~\ref{hip_dif_limitada}. In the table, $L$ is the
    Lipschitz constant associated with $\nabla f$, $\kappa$ is the
    relaxed interpolation condition, $\Lambda$ is associated with the
    geometry of the sample set, $n$ is the dimension of the domain,
    $q + 1$ is the dimension of $\Pol{a}{n}$, $p + 1$ is the number of
    sample points, and $\delta \le \delta_{\max}$ are associated with
    the neighborhood of sampling points. \label{table_error_bounds}}
\end{sidewaystable}

\newpage

\bibliography{library}

\end{document}